\documentclass[11pt]{amsart}
\usepackage{fourier}
\usepackage{amssymb}
\usepackage{amsmath}
\usepackage{amsfonts,latexsym,amstext}
\usepackage [T1]{fontenc}
\usepackage [latin1]{inputenc}
\usepackage{color}
\usepackage{hyperref}

\newtheorem{theorem}{Theorem}[section]
\newtheorem{definition}[theorem]{Definition}

\newtheorem{lemma}[theorem]{Lemma}
\newtheorem{proposition}[theorem]{Proposition}
\newtheorem{corollary}[theorem]{Corollary}
\newtheorem{remark}[theorem]{Remark}

\def\p{\mathcal{P}}
\DeclareMathOperator{\diam}{diam}

\begin{document}

\title{Asymptotic estimates on the von Neumann inequality for homogeneous polynomials}

\author{Daniel Galicer, Santiago Muro, Pablo Sevilla-Peris}

\address{DEPARTAMENTO DE MATEM\'{A}TICA - PAB I,
FACULTAD DE CS. EXACTAS Y NATURALES, UNIVERSIDAD DE BUENOS AIRES, (1428) BUENOS AIRES, ARGENTINA AND CONICET} \email{dgalicer@dm.uba.ar} \email{smuro@dm.uba.ar}

\address{INSTITUTO UNIVERSITARIO DE MATEM\'ATICA PURA Y APLICADA, UNIVERSITAT POLIT\`ECNICA DE DE VAL\`ENCIA, CMNO VERA S/N, 46022, VALENCIA, SPAIN}
\email{psevilla@mat.upv.es}


\thanks{The first two named authors were supported by projects CONICET PIP 0624, PICT 2011-1456,
UBACyT 20020130300057BA, UBACyT20020130300052BA. The third named author was supported by project MTM2014-57838-C2-2-P}


\begin{abstract}
By the von Neumann inequality for homogeneous polynomials there exists a positive constant
$C_{k,q}(n)$ such that for every $k$-homogeneous polynomial $p$ in $n$ variables and every
$n$-tuple of commuting operators  $(T_1, \dots, T_n)$ with $\sum_{i=1}^{n} \Vert T_{i} \Vert^{q} \leq 1$ we have
\[ \|p(T_1, \dots, T_n)\|_{\mathcal L(\mathcal H)} \leq C_{k,q}(n) \; \sup\{ |p(z_1, \dots, z_n)| : \textstyle  \sum_{i=1}^{n} \vert z_{i} \vert^{q} \leq 1 \}\,.
\]
For fixed $k$ and $q$, we study the asymptotic growth of the smallest constant $C_{k,q}(n)$ as $n$ (the number of variables/operators) tends to infinity.
For $q = \infty$, we obtain the correct asymptotic behavior of this constant (answering a question posed by
Dixon in the seventies). For $2 \leq q < \infty$ we improve some lower bounds given by Mantero and Tonge, and
prove the asymptotic behavior up to a logarithmic factor. To achieve this we provide estimates of the norm of
homogeneous unimodular Steiner polynomials, i.e. polynomials such that the multi-indices corresponding to the nonzero
coefficients form partial Steiner systems.
\end{abstract}
\maketitle

\section{Introduction}
A classical inequality in operator theory, due to  von Neumann \cite{vNe51}, asserts that if $T$ is  a linear
contraction on a complex Hilbert space $\mathcal H$ (i.e., its operator norm is less than or equal to one)
then
$$
\|p(T)\|_{\mathcal L(\mathcal H)}\le  \sup\{|p(z)| : z\in \mathbb C, \, |z|\le1 \},
$$
for every polynomial $p$ in one (complex) variable.
Note  that, as a direct consequence of von Neumann's inequality, we can define a functional calculus on the
disk algebra. There are many other consequences of this important inequality in functional analysis; we refer
the reader to \cite[Chapter 1]{Pis01} and the references therein for a fuller treatment of this inequality and
its applications.

For some time, it was very natural to ask whether the von Neumann inequality could  be extended to
polynomials in two or more commuting contractions.
For polynomials in two contractions Ando \cite{And63}, using ``dilation theory'' (see \cite{SzN74}), provided a positive answer. However, in the mid seventies, Varopoulos \cite{Var74} showed that von Neumann's
inequality cannot be extended to three or more contractions. For this, he used the metric theory of tensor
products together with probabilistic tools to construct a polynomial and operators that violate the inequality.
The work of Varopoulos has since been simplified and extended by several authors
\cite{Ble79,CraDav75,Dix76,ManTon79,ManTon80}.

It is an open problem of great interest in operator theory (see \cite{Ble01,Pis01}) to determine whether
there exists a constant $K (n)$ that adjusts von Neumann's inequality. More precisely, it is unknown whether
or not for every $n$ there exists a constant $K(n)$ such that
\begin{equation} \label{Pisier problem}
\|p(T_1,\dots,T_n)\|_{\mathcal L(\mathcal
H)}\le K(n) \; \sup\{ |p(z_1, \dots, z_n)| : |z_i| \leq 1 \},
\end{equation}
for every polynomial $p$ in $n$ variables and every $n$-tuple $(T_1, \dots, T_n)$ of commuting contractions in
$\mathcal L(\mathcal{H})$.

Dixon in \cite{Dix76} gave lower estimates for the optimal $K(n)$ and showed that, if  such a constant
verifying \eqref{Pisier problem} exists, then it must grow faster than any power of $n$. He did  this by
considering the problem in the smaller
class of $k$-homogeneous polynomials.  More precisely, he studied the asymptotic behavior  (as $n$, the number
of variables/operators, tends to infinity) of the smallest constant $C_{k,\infty}(n)$ such that
\begin{equation}\label{Dixon problem}
\|p(T_1,\dots,T_n)\|_{\mathcal L(\mathcal
H)}\le C_{k,\infty}(n) \;  
\sup\{ |p(z_1, \dots, z_n)| : |z_i| \leq 1 \},
\end{equation}
for every $k$-homogeneous polynomial $p$ in  $n$ variables and every $n$-tuple of commuting contractions
$(T_1, \dots, T_n)$. In \cite[Theorem~1.2]{Dix76} he showed that
\begin{equation} \label{wilco}
n^{\frac{1}{2} \big[\frac{k-1}{2} \big]} \ll C_{k,\infty}(n) \ll n^{\frac{k-2}{2}} \,,
\end{equation}
where $[x]$ denotes the integer part of $x$. For the lower bound Dixon used probabilistic techniques (the
Kahane-Salem-Zygmund theorem) and combinatorial ideas (Steiner systems) along with an ingenious construction
of the operators and the Hilbert space involved. \\
This problem was taken up by Mantero and Tonge in  \cite{ManTon79}. Among other problems, for each $1
\leq q < \infty$ they consider $C_{k,q}(n)$, the smallest constant such that
\begin{equation}\label{Mantero Tonge problem}
\|p(T_1,\dots,T_n)\|_{\mathcal L(\mathcal H)}\le C_{k,q}(n) \;   \sup\{|p(z_1,\dots,z_n)| :
\sum_{j=1}^n|z_j|^q \le 1\} \, ,
\end{equation}
for every $k$-homogeneous polynomial $p$ in  $n$ variables and every $n$-tuple of commuting contractions
$(T_1, \dots, T_n)$ with  $\sum_{i=1}^n \|T_i\|_{\mathcal L(\mathcal{H})}^q \leq 1$. They give upper and lower
estimates for the growth of $C_{k,q}(n)$ \cite[Propositions~11 and 17]{ManTon79} (here $q'$ denotes the conjugate of $q$; see below):
\begin{gather}
n^{\frac{k-1}{q'} - \frac{1}{2}\big[ \frac{k}{2} \big] } \ll C_{k,q}(n)  \ll n^{\frac{k-2}{q'}} \text{ for } 1
\leq q \leq 2 \label{q1}, \\
n^{\frac{k}{2} - \frac{1}{2}\big(\big[ \frac{k}{2} \big] +1\big)} \ll C_{k,q}(n)  \ll n^{\frac{k-2}{2}} \text{ for } 2
\leq q < \infty. \label{q2}
\end{gather}
 It is worth noting that the upper bounds here hold for every $n$-tuple $(T_1, \dots, T_n)$ satisfying
$\sum_{i=1}^{n} \Vert T_{i} \Vert^{q} \leq 1$ (and even a weaker condition), not necessarily commuting. If we
do not ask the contractions to
commute, this bound is shown to be optimal in \cite[Proposition~15]{ManTon79}.\\

Based on the combinatorial methods from \cite{Dix76} (i.e., considering polynomials whose monomials are
determined by Steiner blocks) we change the construction of the Hilbert space and the operators  given
there to find the exact asymptotic growth of $C_{k,\infty}(n)$, answering a question that was
explicitly posed by Dixon.\\
On the other hand, by applying some probabilistic tools used by Bayart in \cite{Bay12}, we are able to control
the increments of a Rademacher process and in this way we in this way we manage to narrow the range in \eqref{q2}, showing that the exponent in the power of $n$ is indeed optimal. We collect this in our main result.

\begin{theorem}\label{main}
For $k\ge 3$ and $1\le q\le\infty$, let $C_{k,q}(n)$ be the smallest constant such that
$$
\|p(T_1,\dots,T_n)\|_{\mathcal L(\mathcal H)}\le C_{k,q}(n) \;   \sup\{|p(z_1,\dots,z_n)| : \|(z_j)_j\|_{q}\le
1\},
$$
for every $k$-homogeneous polynomial $p$ in  $n$ variables and every $n$-tuple of commuting contractions
$(T_1, \dots, T_n)$ with  $\sum_{i=1}^n \|T_i\|_{\mathcal L(\mathcal{H})}^q \leq 1$. Then
\begin{enumerate}
	\item \label{vn l_infty} $C_{k,\infty}(n) \sim n^{\frac{k-2}{2}}$
	
	\item \label{vn l_q} for  $2 \leq q < \infty$ we have $$\log^{-3/q}(n) \; n^{\frac{k-2}{2}}  \ll
C_{k,q}(n) \ll n^{\frac{k-2}{2}} \, .$$
	In particular, $n^{\frac{k-2}{2} - \varepsilon } \ll C_{k,q}(n) \ll n^{\frac{k-2}{2}}$ for every
$\varepsilon > 0$.
\end{enumerate}
\end{theorem}

\noindent The proof of this result will be given in Section \ref{prueba}.

\section{Steiner unimodular polynomials}

The systematic study of norms of random homogeneous polynomials started with the Kahane-Salem-Zygmund theorem
\cite[Chapter 6]{Kah68}, which is found very useful in Fourier analysis. More recently,
applications of norms of random  polynomials with unimodular coefficients were found in complex and
functional analysis (see for example \cite{Boa00,DefGarMae04,CarDim07,Bay12}).

The philosophy in this problem and in many others of the same kind (e.g. to compute the Sidon constant for polynomials \cite{MauQue10,DefFreOrtOunSei11}) is to find polynomials which have ``big'' (or ``many'') coefficients, but whose
maximum modulus on the unit ball is ``small''.\\
In this section we are going to  relax the
number of terms appearing in the polynomials, by allowing them to have some zero coefficients. In this way we will find a
special class of tetrahedral unimodular polynomials having many terms, but keeping the maximum modulus quite  small. 

Let us first start with some notation and preliminaries. As usual we will denote $\ell_{q}^{n}$ for $\mathbb{C}^{n}$ with the norm $\Vert (z_{1} , \ldots , z_{n})
\Vert_{q} = \big( \sum_{i=1}^{n} \vert z_{i} \vert^{q} \big)^{1/q}$ if $1 \leq q < \infty$ and $\Vert (z_{1} ,
\ldots , z_{n}) \Vert_{\infty} =
\max_{i=1 , \ldots, n} \vert z_{i} \vert$ for $q=\infty$.\\
 A $k$-homogeneous polynomial in $n$ variables is a function $p : \mathbb{C}^{n} \to \mathbb{C}$ of the form
\[
p(z_{1} , \ldots , z_{n}) = \sum_{\substack{\alpha \in \mathbb{N}_{0}^{n} \\ \vert \alpha \vert =k}}
a_{\alpha} z_{1}^{\alpha_{1}} \cdots z_{n}^{\alpha_{n}}
= \sum_{\substack{J=(j_{1},\ldots , j_{k}) \\ 1 \leq j_{1} \leq \ldots \leq j_{k} \leq n}}
c_{J} z_{j_{1}} \cdots z_{j_{k}}\, ,
\]
 where $a_{\alpha} \in \mathbb{C}$ and $\vert \alpha \vert = \alpha_{1} + \cdots + \alpha_{n} $. Given $\alpha$ we have
 $a_{\alpha}=c_{J}$ where $J=(1, \stackrel{\alpha_{1}}{\ldots}, 1, \ldots , n, \stackrel{\alpha_{k}}{\ldots} ,n)$. We will write $z_{1}^{\alpha_{1}} \cdots z_{n}^{\alpha_{n}} = z^{\alpha}$ and $z_{j_{1}} \cdots z_{j_{k}}=z_{J}$. For $1 \leq q
\leq \infty$ we denote by $\mathcal{P} (^{k}\ell_{q}^{n} )$ the Banach space of all $k$-homogeneous
polynomials on $n$ variables with the norm
\[
\Vert p \Vert_{\mathcal{P} (^{k}\ell_{q}^{n} )} = \sup \{ \vert p(z_{1} , \ldots , z_{n}) \, \colon \, \Vert
(z_{1} , \ldots , z_{n}) \Vert_{q} \leq 1  \} \,.
\]
 It is a well known fact (see e.g. \cite[Chapter~1]{Din99}) that for every $k$-homogeneous polynomial there is
a unique symmetric $k$-linear form $L$ on $\mathbb{C}^{n}$ such that $p(z) = L(z, \ldots , z)$ for all $z \in
\mathbb{C}^{n}$. Also for each $1 \leq q \leq \infty$ and $k \geq 2$ there exists a constant $\lambda(k,q)>0$ such that
\begin{equation} \label{polarizacion}
\Vert p \Vert_{\mathcal{P} (^{k}\ell_{q}^{n} )} \leq \sup \{ L(z^{(1)} , \ldots , z^{(k)}) \colon \Vert
z^{(j)} \Vert_{q} \leq 1 \, , \, j=1, \ldots , k  \} \leq \lambda(k,q) \Vert p \Vert_{\mathcal{P}
(^{k}\ell_{q}^{n} )} \, .
\end{equation}
In general $\lambda(k,q) \leq \frac{k^{k}}{k!}$ but improvements in concrete cases include $\lambda(k,2)=1$ and $\lambda(k,\infty) \leq \frac{k^{\frac{k}{2}} (k+1)^{\frac{k+1}{2}}}{2^{k} k!}$ (see \cite[Propositions~1.44,~1.43]{Din99}).\\
 If $(a_{n})_{n}$ and $(b_{n})_{n}$ are two sequences of real numbers we will write $a_{n} \ll b_{n}$ if there
exists a constant $C>0$ (independent of $n$) such that $a_{n} \leq C b_{n}$ for every $n$. We will write
$a_{n} \sim b_{n}$ if $a_{n} \ll b_{n}$ and $b_{n} \ll a_{n}$.\\
Given a set $A$ we will denote its cardinality by $\vert A \vert$. \\
For an index $1 < q < \infty$ we denote by $q'$ its conjugate: $1 = \frac{1}{q} + \frac{1}{q'}$.

\smallskip

Let $\mathcal C\subset \mathbb N_0^n$ denote any set of multi-indices $\alpha$ with $|\alpha|=k$. Then as a consequence of the
Kahane-Salem-Zygmund theorem \cite[Chapter~6]{Kah68} there exists a $k$-homogeneous polynomial, with
unimodular coefficients $a_\alpha$ for $\alpha\in\mathcal C$ and $a_\alpha=0$ if $\alpha\notin\mathcal C$,
of small maximum modulus on the $n$-polydisk. More precisely, let $(\varepsilon_\alpha)_{\alpha \in \mathcal C}$
be independent Bernoulli variables on a probability space $(\Omega,\Sigma, \mathbb{P})$,
then we have
\begin{equation} \label{prob ellinfty}
\mathbb{P} \{  \omega \in \Omega : \| \sum_{\alpha \in \mathcal C} \varepsilon_\alpha (\omega) z^\alpha
\|_{\p(^k
\ell_\infty^n)} \geq D \big( n \log(k) |\mathcal C| \big)^{1/2} \}  \leq \frac{1}{k^2
e^n },
\end{equation}
where $D >0$ is an absolute constant which is less than 8.
In particular there are signs $(a_\alpha)_{\alpha \in \mathcal C}$ such that the $k$-homogeneous
unimodular polynomial
$$
p(z)= \sum_{\alpha\in\mathcal C}a_\alpha z^\alpha \,,
$$  satisfies
\begin{equation}\label{KSZ}
\|p\|_{\p(^k\ell_\infty^n)}\le D\big( n \log(k) |\mathcal C| \big)^{1/2} \, .
\end{equation}
We are going to work with polynomials with many zero coefficients, expecting that this will make the norm of the polynomial small enough. The presence of  $|\mathcal{C}|$ in  \eqref{KSZ} is sufficient for our needs
when the norm of the polynomial is computed in $\ell_\infty^n$ but not when we consider the norm in  $\ell_p^n$ and then we need different tools. The relevant results we have to hand \cite{Boa00,DefGarMae03,DefGarMae04,Bay12} do not take into account the number of non-zero coefficients, so considering our tetrahedral polynomials does not improve these estimates. We deal with polynomials with a particular combinatorial configuration in order to get useful estimates for our purposes. We modify some arguments from \cite{Bay12}, reflecting this configuration.

To achieve our goal we consider special subsets of multi-indices: partial Steiner systems
on the set $\{1,\dots,n\}$. An $S_p(t, k, n)$ \textit{partial Steiner system} is a collection of subsets of
size $k$ of
$\{1,\dots,n\}$ such that every subset of $t$ elements is contained in at most one member of the
collection of subsets of size $k$.
\begin{definition}\rm
 A $k$-homogeneous polynomial of $n$ variables,  is a \textit{Steiner unimodular polynomial } if
there exists an $S_p(t, k, n)$ partial Steiner system $\mathcal S$ such that
$p(z_{1} , \ldots , z_{n}) = \sum_{J \in \mathcal S} c_{J} z_{J}$ and $c_J=\pm1$. 
\end{definition}
Observe that our Steiner unimodular polynomials are tetrahedral, i.e. in every term $z_{J}$ each variable $z_{j_{0}}$ appears at most once. In other words, no term in the polynomial contains a factor of degree $2$ or
higher in any of the variables $z_{1}, \ldots, z_{n}$.\\
The first one to consider Steiner unimodular polynomials was Dixon \cite{Dix76}, who used $S_{p}([(k-1)/2],k,n)$ partial Steiner systems. He used this to obtain lower bounds for \eqref{Dixon problem}. The combinatorial property was only applied to define some Hilbert space operators that violate the inequality, but not to estimate the norm of the polynomial, which he did using \eqref{KSZ} and the number of non-zero coefficients.

In the following lemmas, in $\ell_q^n$, $1 \leq q<\infty$, we will
strongly use the fact that the multi-indices of the non-zero coefficients form a partial Steiner system to
estimate the maximum modulus.
We use an entropy argument due to Pisier to control the
increments of a Rademacher process and subsequently apply an interpolation argument.

\medskip
Let us first recall some definitions and a result on regularity of random process. A complete account on these can
be found in \cite[Chapters~4 and 11]{LedTal91}. \\
A Young function $\psi$ is a convex increasing function defined on $[0, \infty[$ such that $\lim_{t\to
\infty}\psi(t)=\infty$ and $\psi(0)=0$. For a probability space $( \Omega, \Sigma, \mathbb{P})$, the  Orlicz
space $L_\psi = L_\psi ( \Omega, \Sigma, \mathbb{P})$  is defined as the space of all real-valued random
variables $Z$ for which there exists $c>0$ such that $\mathbb{E}(\psi(|Z|/c)) < \infty$. It is a Banach space
with the norm $\|Z\|_{L_\psi} = \inf \{ c >0 : \mathbb{E}(\psi(|Z|/c)) \leq 1 \}$.\\
Let $(X,d)$ be a metric space. Given $\varepsilon >0$, the entropy number $N(X,d;\varepsilon)$ is defined as
the smallest number of open balls of radius $\varepsilon$ in the metric $d$, which form a covering of the
metric space $X$.\\
With this, the entropy integral of $(X,d)$ with respect to $\psi$ is given by
\[
J_\psi(X,d) : = \int_0^{\diam(X)} \psi^{-1}(N(X,d;\varepsilon)) d\varepsilon \, .
\]
\medskip

\noindent We are going to define a random process $(Y_{z})_{z \in B_{\ell_{2}^{n}}}$ and we will need
to estimate the expectation of $\sup_{z} Y_{z}$. To do so, we use the following theorem due to
Pisier \cite{Pis83} (see also \cite[Theorem~11.1]{LedTal91}) that bounds this expectation with the entropy
integral, provided that the random process satisfies a certain contraction condition.

\begin{theorem}\label{Pisier's result} Let $Z= (Z_x)_{x \in X}$ be a random process indexed by $(X,d)$ in
$L_\psi$ such that, for every $x,x' \in X$,
\[
\|Z_x - Z_{x'}\|_{L_\psi} \leq d(x,x') \,.
\]
Then, if $J_\psi(X,d)$ is finite, $Z$ is almost surely bounded and
\[
\mathbb{E} \big( \sup_{x,x'\in X} |Z_x - Z_{x'}| \big) \leq 8 J_\psi(X,d) \, .
\]
\end{theorem}
Let now $k \geq 2$ and let $\mathcal S$ be a $S_p(k-1,k,n)$ partial Steiner system. 
We consider  a family of independent Bernoulli variables $(\varepsilon_J)_{J\in\mathcal S}$ on the probability space
$(\Omega, \Sigma,\mathbb{P})$. For $z \in B_{\ell_2^n}$ we define the following Rademacher process indexed by
$B_{\ell_2^n}$ as
\begin{equation} \label{benvolgut}
Y_z=\frac{1}{k}\sum_{J\in\mathcal S}\varepsilon_J z_{J}.
\end{equation}
We view it as a random process in the Orlicz space
defined by the Young function $\psi_{2}(t)  =  e^{t^2} -1$.
\begin{lemma}\label{Rademacher process}
The Rademacher process defined in \eqref{benvolgut} fulfils the following Lipschitz condition: 
$$
\|Y_z - Y_{z'}\|_{L_{\psi_2}} \leq C  \| z - z' \|_{\infty},
$$
for some universal constant $C\geq 1$ and every $z,z' \in B_{\ell_2^n}$.
\end{lemma}
\begin{proof}
As a consequence of Khintchine inequalities (see e.g. \cite[Sect~8.5]{DefFlo93}), the $\psi_2$-norm of a
Rademacher process is comparable to its $L_2$-norm. Now,
\begin{align*}
\|Y_z - Y_{z'}\|_{L_2} & =  \frac{1}{k} \big( \int_{\Omega} \big|\sum_{J\in\mathcal S}\varepsilon_J(\omega)
(z_J - {z}_J') \big|^2 d\mathbb{P}(\omega)\big)^{1/2}
 = \frac{1}{k} \big( \sum_{J\in\mathcal S}\  \big|z_J - z_J' \big|^2 \big)^{1/2} \\
& = \frac{1}{k} \big( \sum_{J\in\mathcal S}\  \big| \sum_{u=1}^k z_{j_1} \dots z_{j_{u-1}} (z_{j_u} -
z_{j_u}') z_{j_{u+1}}' \dots  z_{j_{k}}' \big|^2 )^{1/2} \\
& \leq \frac{1}{k} \sum_{u=1}^k \big( \sum_{J\in\mathcal S}\  \big|  z_{j_1} \dots z_{j_{u-1}} (z_{j_u} -
z_{j_u}') z_{j_{u+1}}' \dots  z_{j_{k}}' \big|^2 \big)^{1/2} \\
& \leq \frac{1}{k} \sum_{u=1}^k  \| z - z' \|_{\infty} \big( \sum_{J\in\mathcal S}\  \big|  z_{j_1}
 \dots z_{j_{u-1}} z_{j_{u+1}}' \dots  z_{j_{k}}' \big|^2 \big)^{1/2}
\end{align*}
Since $\mathcal S$ is an $S_p(k-1,k,n)$ partial Steiner system, given $j_1, \dots,
j_{u-1},j_{u+1},\dots,j_k$ for a fixed $u$, there is at most one index $j_u$ such that $(j_1, \dots, j_{k})$
belongs to
$\mathcal S$. Therefore the sum $\sum_{J\in\mathcal S}\  \big|  z_{j_1} \dots
z_{j_{u-1}} z_{j_{u+1}}' \dots  z_{j_{k}}' \big|^2$ can be bounded by 
$$
(\sum_{l_1=1}^n
|z_{l_1}|^2 ) \cdots (\sum_{l_{u-1}=1}^n  |z_{l_{u-1}}|^2) (\sum_{l_{u+1}=1}^n  |z_{l_{u+1}}'|^2 ) \cdots
(\sum_{l_{k}=1}^n |z_{l_{k}}'|^2 ) ,
$$ 
and this is less than or equal to one (since $z, z' \in B_{\ell_2^n}$).
This combined with the previous inequality concludes the proof.
\end{proof}
We are now in a position to use Theorem~\ref{Pisier's result} with $L_{\psi_{2}}$, $X=B_{\ell_{2}^{n}}$ and
$d=\Vert \cdot\Vert_{\infty}$ to bound the expectation of the supremum. For this, we estimate the entropy
integral
 $J_{\psi_{2}} (B_{\ell_2^n},\|~\cdot~\|_\infty)$. Note that $\psi_{2}^{-1}(t)=\log^{1/2} (t+1)$; we use
instead $\log^{1/2} (t)$, which does not change the computation of the integral. We estimate the integral in the
following result, which is a version of
 \cite[Lemma~2.1]{Bay12}; the proof is essentially the same and we include it here for the sake of
completeness.

\begin{lemma}\label{entropy control}
There exists $C>0$ such that for every $n \geq 2$ we have
$$J_{\psi_2}(B_{\ell_2^n},\|~\cdot~\|_\infty) \leq C \log^{3/2}(n) \,.$$
\end{lemma}
\begin{proof}
We fix $n$ and for each $m$ we consider the number
$$e_m = \inf \{ \sigma > 0 : B_{\ell_2^n} \subset \bigcup_{i=1}^{2^m}x_i + \sigma B_{\ell_\infty^n}\}\,.$$
By result of Sch\"utt \cite[Theorem~1]{Sch84} there exists a constant $K$, independent of $n$ and $m$, such
that
\begin{equation}\label{Sch84}
e_m \leq K \times \left\{\begin{array}{cl}
1 & \mbox{if } m \leq \log(n),\\
\big(\frac{\log(1 + \frac{2n}{m})}{m} \big)^{1/2} & \mbox{if } \log(n) \leq m \leq 2n, \\
2^{-\frac{m}{2n}} n^{-\frac{1}{2}} & \mbox{if } m \geq 2n \, .
\end{array}
\right.
\end{equation}
Let us note that Sch\"utt's result is stated for real spaces. Since the $(2n)$-dimensional real euclidean
space is isometrically isomorphic to $\ell_{2}^{n}$ we get \eqref{Sch84}.\\
For $m\geq 2n$, if $K \, 2^{-\frac{m+1}{2n}} n^{-\frac{1}{2}} \leq \varepsilon < K 2^{-\frac{m}{2n}}
n^{-\frac{1}{2}}$, then by \eqref{Sch84} we have $N(B_{\ell_2^n},\|~\cdot~\|_\infty;\varepsilon)
\leq 2^{m+1} \leq 2 \frac{K^{2n}}{\varepsilon^{2n} n^{n}}$ and
\[
\int_{0}^{K/(2 \sqrt{n})}  \log^{1/2}\big( N(B_{\ell_2^n},\|~\cdot~\|_\infty;\varepsilon) \big) d
\varepsilon
\leq \int_0^{K/(2 \sqrt{n})} n^{1/2} \log^{1/2} \big( \frac{2K^2}{\varepsilon^2 n} \big) d\varepsilon
 = \int_0^{K/2} \log^{1/2} \big( \frac{2C^2}{u^2 } \big) du =K_{1}< \infty \,.
\]
With the same argument, if $\frac{K}{2 \sqrt{n}} \leq \varepsilon K \sqrt{\frac{\log 2}{2n}}$ then
$N(B_{\ell_2^n},\|~\cdot~\|_\infty;\varepsilon) \leq 2^{2n}$ and with this we can bound the integral from
$\frac{K}{2 \sqrt{n}}$ to $K \sqrt{\frac{\log 2}{2n}}$ by some $K_{2}$.\\
We define now $\varepsilon_{m} = \big(\frac{\log(1 + \frac{2n}{m})}{m} \big)^{1/2}$ for $[\log n] \leq m< 2n$.
Again by \eqref{Sch84}, if $\varepsilon_{m+1} \leq \varepsilon < \varepsilon_{m}$ then
$N(B_{\ell_2^n},\|~\cdot~\|_\infty;\varepsilon) \leq 2^{m+1}$. Then
\[
\int_{\varepsilon_{2n}}^{\varepsilon_{[\log(n)]}} \log^{1/2} \big(
N(B_{\ell_2^n},\|~\cdot~\|_\infty;\varepsilon) \big) d\varepsilon \leq \sum_{m=[\log(n)]}^{2n-1} (m +1)^{1/2}
(\varepsilon_m - \varepsilon_{m+1}) \log^{1/2}(2)\,.
\]
We write
\[
(\varepsilon_m - \varepsilon_{m+1})  =  K\bigg[ \frac{\log^{1/2}(1 + \frac{2n}{m})}{m^{1/2}} -
\frac{\log^{1/2}(1 + \frac{2n}{m})}{(m+1)^{1/2}} + \frac{\log^{1/2}(1 + \frac{2n}{m})}{(m+1)^{1/2}}
 - \frac{\log^{1/2}(1 + \frac{2n}{m+1})}{(m+1)^{1/2}} \bigg]
\]
and we get
\begin{multline*}
\int_{\varepsilon_{2n}}^{\varepsilon_{[\log(n)]}} \log^{1/2} \big(
N(B_{\ell_2^n},\|~\cdot~\|_\infty;\varepsilon) \big) d\varepsilon  \\
\leq   K \bigg( \log^{1/2}(n)  \sum_{s=[\log(n)]}^{2s-1} \frac{(s +1)^{1/2}}{s^{3/2}}
 + \sum_{s=[\log(n)]}^{2s-1} \log^{1/2}(1 + \frac{2n}{s}) - \log^{1/2}(1 + \frac{2n}{s+1})  \bigg)
  \leq K_{3} \log^{3/2}(n).
\end{multline*}
Finally, for the remaining subinterval we have that, by  \eqref{Sch84}, if $\varepsilon \geq
\varepsilon_{[\log(n)]}$, then $N(B_{\ell_2^n},\|~\cdot~\|_\infty;\varepsilon) \leq 2^{\log(n)}$. Hence
\[
\int_{\varepsilon_{[\log(n)]}}^1 \log^{1/2} \big( N(B_{\ell_2^n},\|~\cdot~\|_\infty;\varepsilon) \big)
d\varepsilon \leq  K_{4} \int_0^1 \log^{1/2}(n) d\varepsilon \leq  K_{4} \log^{1/2}(n) \, .
\]
This completes the proof.
\end{proof}

We can now find Steiner unimodular polynomials that have small norm in $\p(^k\ell_q^n)$, for every  $2 \leq q \le\infty$ simultaneously.
\begin{theorem}\label{Steiner con norma chica}
Let $k\ge 2$  and $\mathcal S$ be an $S_p(k-1,k,n)$ partial Steiner system. Then there exist signs
$(c_J)_{J\in
\mathcal S}$ and a constant $A_{k,q}>0$ independent of $n$  such that the
$k$-homogeneous polynomial $p = \sum_{J\in\mathcal S} c_J z_J$ satisfies
\[
\|p\|_{\p(^k\ell_q^n)} \leq A_{k,q} \times \begin{cases} \log^{\frac{3}{q}}(n) n^{\frac{k}{2} (\frac{q-2}{q})} & \text{ for } 2 \leq q <\infty \\
\log^{\frac{3q-3}{q}}(n) & \text{ for } 1\le q\le 2 \,.
\end{cases}
\]
Moreover, the constant $A_{k,q}$ may be taken independent of $k$ for $q \neq 2$.
\end{theorem}
\begin{proof}

To prove this theorem we will first find a polynomial with small norm both in ${\p(^k\ell_2^n)}$ and in
${\p(^k\ell_\infty^n)}$. For this we use an interesting technique borrowed from the proof of \cite[Lemma~2.1]{CarDim07}, followed by an interpolation argument.

Note first that any $S_p(k-1,k,n)$ partial Steiner system $\mathcal S$ satisfies that $|\mathcal S| \leq \frac{1}{k}\binom{n}{k-1}.$
We use $\mathcal S$ to define a Rademacher process  $(Y_z)_{z \in
B_{\ell_2^n}}$ as in \eqref{benvolgut}.
By Lemma~\ref{Rademacher process}, Theorem~\ref{Pisier's result} and Lemma~\ref{entropy control} there is a
constant $K>0$ such that $\mathbb{E} (\sup_{z \in B_{\ell_2^n}} |Y_z|) \leq K \log^{3/2}(n)$. Therefore, by
Markov's inequality we have
\begin{equation} \label{prob ell2}
\mathbb{P} \{ \omega \in \Omega : \|\sum_{J\in\mathcal S}\varepsilon_J (\omega) z_J \|_{\p(^k\ell_2^n)} \geq
M k K \log^{3/2}(n) \} \leq \frac{1}{M},
\end{equation}
where $M$ is some constant to be determined.
On the other hand, recall that by \eqref{prob ellinfty} we have
$$\mathbb{P} \{  \omega \in \Omega : \| \sum_{J \in \mathcal S} \varepsilon_J (\omega) z_J \|_{\p(^k
\ell_\infty^n)} \geq D \big( n \log(k) |\mathcal S| \big)^{1/2} \}  \leq \frac{1}{k^2 e^n},$$
Therefore, if $M >1+ \frac{1}{k^2e^n -1}$ (note that we can take $M=2$ here) we have the following inequalities for $\omega$ in a positive
measure set
\begin{equation} \label{acotacion normas}
\left\{
\begin{array}{cl}
 \|\sum_{J\in\mathcal S}\varepsilon_J (\omega) z_J \|_{\p(^k\ell_2^n)} & \leq M k K
\log^{3/2}(n), \\
 \| \sum_{J \in \mathcal S} \varepsilon_J (\omega) z_J \|_{\p(^k \ell_\infty^n)} & \leq D \big( n \log(k)
|\mathcal S| \big)^{1/2} \leq D  \Big(\frac{\log(k)}{k}  \binom{n}{k-1}n\Big)^{1/2} \leq D \Big(\frac{\log(k)}{k!} n^{k}\Big)^{1/2}.
\end{array}
\right.
\end{equation}
There is a choice of signs $(c_J)_{J \in \mathcal S}$ such that the polynomial $p(z):=
\sum_{J\in\mathcal S}c_J  z_J$ satisfies the inequalities in \eqref{acotacion normas}.
We now use an interpolation argument to obtain a bound of the norm of $p$ in ${\p(^k\ell_q^n)}$ for
$2<q<\infty$. We consider
the $k$-linear form associated to $p$ then \cite[Theorem~4.4.1]{BerLof76}, together with \eqref{polarizacion}
and \eqref{acotacion normas}, give
\begin{align}
\|p\|_{\p(^k\ell_q^n)} &\leq   \big(Mk K\big)^{2/q} \big(D \lambda(k,\infty)
\frac{\log^{1/2}(k)}{\sqrt{k!}}\big)^{\frac{q-2}{q}}  \; \log^{3/q}(n) n^{\frac{k}{2}
(\frac{q-2}{q})}\\
& \le  \underbrace{\max\{MK,D\} \Big( \frac{k^{\frac{k}{2}} (k+1)^{\frac{k+1}{2}} \sqrt{\log k}}{2^{k} k! \sqrt{k!}} \Big)^{\frac{q-2}{q}}k^{\frac{2}{q}}}_{A_{k,q}}\; \log^{3/q}(n) n^{\frac{k}{2}
(\frac{q-2}{q})} \,.
\end{align}

Note that for $q>2$, $A_{k,q}\to 0$ as $k\to\infty$, and thus we may take a constant independent of $k$ in
this case.

For $q=1$, it is immediately seen that every Steiner unimodular polynomial has norm less than or equal to one.
Actually,  more can be said. Let $P(z)=\sum_{|\alpha|=k}a_\alpha z^\alpha$ be any $k$-homogeneous
polynomial. Then
\begin{align}\label{pols en l1}
 |P(z)| &\le   \sum_{|\alpha|=k}|a_\alpha z^\alpha| \,\le\,
\sup_{|\alpha|=k}\Big\{|a_\alpha|\frac{\alpha!}{k!}\Big\}\sum_{|\alpha|=k}|\frac{k!}{\alpha!} z^\alpha| =
\sup_{|\alpha|=k}\Big\{|a_\alpha|\frac{\alpha!}{k!}\Big\}\Big(\sum_{j=1}^n|z_j|\Big)^k.
\end{align}
In particular, the polynomial $p$ considered above satisfies $\|p\|_{\p(^k\ell_1^n)}\le \frac{1}{k!}$.
Finally, proceeding by interpolation between the $\ell_1^n$ and $\ell_2^n$ cases we obtain that for $1<q<2$,
\[
\|p\|_{\p(^k\ell_q^n)} \leq  \Big( \frac{k^k}{(k!)^{2}} \Big)^{\frac{2-q}{q}} \big(Mk K\log^{3/2}(n)\big)^{\frac{2q-2}{q}} = A_{k,q} \log^{\frac{3q-3}{q}}(n)\,.
\]
Note that also in this case, for every $1 \leq q < 2$ we have $A_{k,q} \to 0$ as $k \to \infty$.
\end{proof}
As was already noted in \cite[Corollary~6.5]{DefGarMae03}, the argument in \eqref{pols en l1} improves the estimates given in
\cite{Boa00} and \cite[Corollary~3.2]{Bay12} for the $q=1$ case.
\begin{remark}\rm
It is not difficult to prove that every 2-homogeneous Steiner unimodular polynomial has norm in
$\p(^2\ell_2^n)$ less than or equal to $\frac12$. It would be interesting to know if there
exists a constant $C$, perhaps depending on $k \geq 3$ and not on $n$,  such that given any $S_p(k-1,k,n)$ partial
Steiner system $\mathcal S$, we can find a $k$-homogeneous unimodular polynomial $p(z):=
\sum_{J\in\mathcal S}c_J  z_J$ with $\|p\|_{\p(^k\ell_2^n)}\le C$.
An affirmative answer to this question would in particular give that the upper bound given by Mantero and Tonge \eqref{q2} for $C_{k,q}(n)$ with $2 \leq q < \infty$  is actually optimal.
\end{remark}

The last ingredient we need for our applications  is the existence of nearly
optimal partial Steiner systems, in the sense that they have many elements. This translates to many unimodular
coefficients of the Steiner polynomials. It is well known that any partial Steiner system $S_p(t,k,n)$ has
cardinality less than or equal to $\binom{n}{t}/\binom{k}{t}$.
A conjecture of Erd\H os and Hanani \cite{ErdHan63}, proved positively by R\"odl \cite{Rod85}, states that
there
exist partial Steiner systems $S_p(t,k,n)$ of cardinality at least $(1-o(1))\binom{n}{t}/\binom{k}{t}$, where
$o(1)$ tends to zero as $n$ goes to infinity. This bound was improved in \cite{AloKimSpe97} (see also \cite{Kim01} for a panoramic overview of the subject), where it is proved  that
there exists a constant $c>0$ such that there exist partial Steiner systems $S_p(k-1,k,n)$ of cardinality  at least
\begin{equation}\label{Alon}
\begin{split}
 \frac{\binom{n}{k-1}}{k}\Big(1-\frac{c}{n^{\frac1{k-1}}}\Big), \quad \textrm{ for }k>3\, ,\\
\frac{\binom{n}{k-1}}{k}\Big(1-\frac{c\log^{3/2}n}{n^{\frac1{k-1}}}\Big), \quad \textrm{ for }k=3\,.
\end{split}
\end{equation}
Taking partial Steiner systems of this cardinality in Theorem \ref{Steiner con norma chica} we have the
following.
\begin{corollary}
 Let $k\ge 3$. Then there exists a $k$-homogeneous Steiner
unimodular polynomial $p$ of $n$ complex variables with at least $\psi(k,n)$ (defined in \eqref{Alon}) coefficients satisfying the estimates in Theorem~\ref{Steiner con norma chica}.
Note that in this case $\psi(k,n) \gg n^{k-1}$.
\end{corollary}
\begin{remark}\rm \label{existen Steiner}
Very recently, a longstanding open problem in combinatorial design theory was solved by Keevash \cite{Kee}.
 A \textit{Steiner system} $S(t, k, n)$ is a collection of subsets of size $k$ of $\{1,\dots,n\}$ such that
every subset of $t$ elements is contained in exactly one member of the collection of subsets of size $k$.
Keevash's result implies the asymptotic existence of Steiner systems, that is, that given $t<k$, Steiner
systems  $S(t, k, n)$ exist for every sufficiently large $n$ that satisfies some natural divisibility conditions. In
particular, for an infinite number of $n$'s  we may take $\psi(k,n)=\binom{n}{k-1}/{k}$ in the above corollary.
\end{remark}

\section{Estimates on the multivariable von Neumann inequality} \label{prueba}

In this section we estimate the asymptotic failure of different versions of the multivariable von Neumann
inequality for homogeneous polynomials. Before we prove Theorem \ref{main}, let us observe that we modify Dixon's original proof of the lower bound in \eqref{wilco} in several ways.\\
Dixon considered partial Steiner systems $S_{p}([(k-1)/2],k,n)$, for which the number of non-zero coefficients is of the order $n^{[\frac{k-1}{2}]}$. This is not enough to find a good lower bound. Instead, we use partial Steiner systems $S_{p}(k-1,k,n)$. This allows us to have more non-zero coefficients, but also forces us to make a new construction of the Hilbert space and the operators which we feel is closer to that given by Varopoulos in \cite{Var74}.

\begin{proof}[Proof of Theorem~\ref{main}--(\ref{vn l_infty})]
The upper bound was proved in \cite[Theorem~1.2]{Dix76}. Thus  we only have to construct a polynomial, a Hilbert space and
commuting contractions that show that the asymptotic growth of this bound is optimal.\\
Let $n \geq k\ge 3$ and choose a partial Steiner system $S_{p}(k-1,k,n)$, denoted by $\mathcal S$, such that $|\mathcal S| = \psi(k,n)$ as in \eqref{Alon}. By Theorem~\ref{Steiner con norma chica}, see also \eqref{acotacion normas},
there exists a $k$-homogeneous polynomial $p(z)= \sum_{J\in\mathcal S}c_Jz_J,$ with $c_J=\pm1$
for every $J\in\mathcal S$ and such that
\begin{equation}\label{KSZ2}
\|p\|_{\p(^k\ell_\infty^n)}\leq D  \Big(\frac{\log(k)}{k}  \binom{n}{k-1}n\Big)^{1/2} \, .
\end{equation}
Let $\mathcal H$ be the (finite dimensional) Hilbert space which has as orthonormal basis the
following vectors
$$
\left\{
\begin{array}{cl}
 e; & \\
 e(j_1,\dots,j_m) & \textrm{ for } 0\le m\le k-2 \textrm{ and } 1\leq j_1\le\dots\le j_m\leq n;\\
 f_i & \textrm{ for } i=1,\dots,n;\\
 g. &
\end{array}
\right.
$$
Given any subset $\{i_1,\dots,i_r\}\subset\{1,\dots,n\}$, we denote by $[i_1,\dots,i_r]$ its nondecreasing
reordering.\\
We define, for $l=1,\dots,n$, the operators that act as follows on the basis of $\mathcal H$,
$$
\begin{array}{cl}
 T_le=e(l) & \\
 T_le(j_1,\dots,j_m)=e[l,j_1,\dots,j_m], & \textrm{ if } 0\le m< k-2\\
 T_le(j_1,\dots,j_{k-2})=\sum_i \gamma_{\{i,l,j_1,\dots,j_{k-2}\}}f_i, & \\
 T_lf_i=\delta_{li}g, & \\
 T_lg=0, &
\end{array}
$$
where
$$
\gamma_{\{i_1,\dots,i_{k}\}}=\left\{\begin{array}{cl}
 c_{\{i_1,\dots,i_{k}\}}, & \text{ if } \{i_1,\dots,i_{k}\}\in\mathcal S\\
 0 & \text{ otherwise.}
\end{array}\right.
$$
Since $\mathcal S$ is an $S_{p}(k-1,k,n)$ partial Steiner system, $\|T_l\|=1$ for $l=1,\dots,n$. It is easily checked that
the
operators commute. We have
$$
p(T_1,\dots,T_n)e=\sum_{\{i_1,\dots,i_{k}\}\in\mathcal
S}c_{\{i_1,\dots,i_{k}\}}T_{i_1}T_{i_2}\dots T_{i_k}e = \sum_{\{i_1,\dots,i_{k}\}\in\mathcal
S}c_{\{i_1,\dots,i_{k}\}}^2g=|\mathcal S|g =\psi(k,n) g.
$$
Now, using \eqref{KSZ2} we get
\begin{align*}
 \|p(T_1,\dots,T_n)\|_{\mathcal L(\mathcal H)}  & \ge  \|p(T_1,\dots,T_n)  e\|_{\mathcal H}
= \psi(k,n)  \\
& \geq \frac1{D} \Big(\frac{\binom{n}{k-1}}{nk\log(k)}\Big)^{\frac12} \big(1-o(1) \big)\|p\|_{\p(^k\ell_\infty^n)}
\gg n^{\frac{k-2}{2}}\|p\|_{\p(^k\ell_\infty^n)}.
\end{align*}
This gives the desired conclusion.
\end{proof}

\begin{proof}[Proof of Theorem~\ref{main}--\eqref{vn l_q}.]
The upper bound was proved in  \cite[Corollary~11]{ManTon79}. For the lower bound, we take the Hilbert space and the operators $T_1,\dots,T_n$ defined in the proof of
Theorem~\ref{main}--\eqref{vn l_infty}. Then $R_{j}= \frac{T_j}{n^{1/q}}$ for $j=1, \ldots ,n$ clearly satisfy
$\sum_{i=1}^{n} \Vert R_{i} \Vert^{q} \leq 1$. Taking
the polynomial $p$ given by Theorem~\ref{Steiner con norma chica} we have
\begin{align*}
 \|p(R_1, \dots, R_n)\|_{\mathcal L(\mathcal H)}  & \ge  \frac{1}{n^{k/q}} \|p(T_1,\dots,T_n)e\|_{\mathcal H}
= \frac{|\mathcal S|}{n^{k/q}}  \\
 & \ge \frac{\|p\|_{\p(^k\ell_q^n)} |\mathcal S| }{ A_{k,q} \log^{3/q}(n) n^{\frac{k}{2} (\frac{q-2}{q})} \;
n^{k/q}}
\ \ge \
A_{k,q}^{-1} C_k \log^{- 3/q}(n) n^{\frac{k-2}{2}}\|p\|_{\p(^k\ell_2^n)}.
\end{align*}
This concludes the proof of the theorem.
\end{proof}

\subsection{Other possible extensions of the von Neumann inequality for homogeneous polynomials: some particular cases}

Mantero and Tonge \cite[Proposition 17]{ManTon79} also obtained lower bounds for $C_{k,q,r}(n)$,
defined as the least constant $C$ such that
\begin{equation}
\|p(T_1,\dots,T_n)\|_{\mathcal L(\mathcal H)}\le C\;   \sup\{|p(z_1,\dots,z_n)| :
\sum_{j=1}^n|z_j|^q \le 1\} \, ,
\end{equation}
for every $k$-homogeneous polynomial $p$ in  $n$ variables and every $n$-tuple of commuting contractions
$(T_1, \dots, T_n)$ with  $\sum_{i=1}^n \|T_i\|_{\mathcal L(\mathcal{H})}^r \leq 1$.
Proceeding as in the proof of Theorem~\ref{main}--(\ref{vn l_q}), we can show the following.
\begin{proposition}\label{vn q,r}
Let $k\ge 3$, then
\begin{enumerate}
	\item  \label{vn q,r >2} $
\log^{-3/q}(n)n^{k(\frac{1}{2}+\frac{1}{q}-\frac{1}{r})-1} \ll C_{k,q,r}(n)$, for $q\ge 2$ and $1\le r\le\infty$,
	
	\item  \label{vn q,r <2} $
\log^{-3/q'}(n)n^{\frac{k}{r'}-1} \ll C_{k,q,r}(n)$, for $q\le 2$ and $1\le r\le\infty$.
\end{enumerate}
\end{proposition}
\begin{remark}\rm
The above proposition improves the lower bounds for $C_{k,q,r}$ given in \cite[Proposition~17]{ManTon79}
in all cases but $q\le2$ and $k=3$.
\end{remark}

Another possible multivariable extension of the von Neumann inequality (also studied in \cite{ManTon79}) is
by considering polynomials on commuting operators $T_1,\dots,T_n$ satisfying that for any
pair $h,g$ of norm one vectors in the Hilbert space,
\begin{equation}\label{IVp}
\sum_{j=1}^n|\langle T_jh,g\rangle|^q\le1,
\end{equation}
or, equivalently, that for any vector $\alpha\in\mathbb C^n$ such that $\|\alpha\|_{\ell_{q'}^n}=1$, we have
$$
\Big\|\sum_{j=1}^n\alpha_jT_j\Big\|\le1.
$$
Let $D_{k,q}(n)$,
denote the smallest constant such that
\begin{equation}
\|p(T_1,\dots,T_n)\|_{\mathcal L(\mathcal H)}\le D_{k,q}(n)\;   \sup\{|p(z_1,\dots,z_n)| :
\sum_{j=1}^n|z_j|^q \le 1\} \, ,
\end{equation}
for every $k$-homogeneous polynomial $p$ in  $n$ variables and every $n$-tuple of commuting contractions
$(T_1, \dots, T_n)$ satisfying \eqref{IVp}.
The upper bound obtained in \cite[Proposition~20]{ManTon79} is
$$
D_{k,q}(n)\ll
\left\{\begin{array}{ll}
                         n^{(k-1)(\frac12+\frac1{q})} & \textrm{ for } q\ge2,\\
                         n^{(k-1)(\frac12+\frac1{q'})} & \textrm{ for } q\le2.
                        \end{array}\right.
$$
For $k=3$ and $q=2$ we show that this is optimal up to a logarithmic factor.
\begin{proposition}
We have the following asymptotic behavior: $\displaystyle\frac{n^2}{\log^{15/4} n}\ll D_{3,2}(n)\ll n^2.$
\end{proposition}

\begin{proof}
 Let $p(z) = \sum_{J\in\mathcal S} c_J z_J$ be a 3-homogeneous Steiner unimodular polynomial as in Theorem~\ref{Steiner con norma chica} and let
$T_1,\dots,T_n$ be the operators defined in the proof of
Theorem~\ref{main}--\eqref{vn l_infty}.
We prove first that
$\frac{T_1}{\|p\|^{1/2}_{\p(^3\ell^n_2)}},\dots,\frac{T_n}{\|p\|^{1/2}_{\p(^3\ell^n_2)}}$ satisfy \eqref{IVp}. Note that these operators are defined on a $(2n+2)$-dimensional Hilbert
space $\mathcal{H}$ with orthonormal basis $\{e,e_1,\dots,e_n,f_1,\dots,f_n,g\}$.

For $\alpha\in\ell_{2}^n$ and $h\in \mathcal{H}$, (below we take some $\beta$ in the unit ball of $\ell_{2}^n$)
\begin{align*}
\Big\|\sum_j\alpha_jT_jh\Big\|^2 &=\sum_j|\alpha_j\langle h,e\rangle|^2 +\sum_i\Big|\sum_{j,l}\alpha_j\langle
h,e_l\rangle a_{\{i,j,l\}}\Big|^2 +\Big|\sum_j\alpha_j\langle h,f_j\rangle\Big|^2\\
& = |\langle h,e\rangle|^2\|\alpha\|_{\ell_2^n}^2 + \Big( \sum_{i} \beta_{i} \sum_{j,l}\alpha_j\langle
h,e_l\rangle a_{\{i,j,l\}}  \Big)^{2}
 +\Big|\sum_j\alpha_j\langle h,f_j\rangle\Big|^2 \\
&\le |\langle h,e\rangle|^2\|\alpha\|_{\ell_2^n}^2 +\|p\|_{\p(^3\ell^n_2)}\|\alpha\|_{\ell_2^n}^2\|(\langle
h,e_l\rangle)_l\|_{\ell_2^n}^2 +\|\alpha\|_{\ell_2^n}^2\|(\langle h,f_j\rangle)_j\|_{\ell_2^n}^2\\
&\le \|p\|_{\p(^3\ell^n_2)}\|\alpha\|_{\ell_2^n}^2\|h\|_{\mathcal H}^2.
\end{align*}
Therefore,
\begin{align*}
 \Big\|p(\frac{T_1}{\|p\|^{1/2}_{\p(^3\ell^n_2)}}, \dots, \frac{T_n}{\|p\|^{1/2}_{\p(^3\ell^n_2)}})\Big\|_{\mathcal
L(\mathcal H)}  & \ge  \|p\|_{\p(^3\ell^n_2)}^{-3/2}\|p(T_1,\dots,T_n)e\|_{\mathcal H} \\
& = \|p\|_{\p(^3\ell^n_2)}^{-3/2}|\mathcal S|  \gg \|p\|_{\p(^3\ell^n_2)} \frac{n^2}{\log^{15/4} n},
\end{align*}
and this concludes the proof.
\end{proof}

\subsection*{Acknowledgements} We wish to thank the referee for her/his comments. Also, we would like to warmly thank our friend Michael Mackey for his careful reading and suggestions that improved considerably the
final presentation of the paper.

\end{document}